\pdfoutput=1

  \documentclass[11pt,letterpaper,twoside]{article}

  \usepackage[margin=2.9cm]{geometry}
  \usepackage{titlesec}
  \titlelabel{\thetitle.\,\,\,}
  \titlespacing*{\section}{0pt}{1.1em}{0.3em}
  \titlespacing*{\subsection}{0pt}{.5em}{0pt}

  \usepackage{amsmath,amssymb,amsthm,amsfonts,bbding,stmaryrd,dsfont,wasysym,pifont,xspace,xcolor}
  \usepackage{parskip,fancyhdr,url}

  \usepackage{enumitem}
  \usepackage{fourier}

  \usepackage{epstopdf,yfonts}
  \usepackage{graphicx,epsfig} 
  \usepackage{mathrsfs}

  \usepackage[
     breaklinks,colorlinks,
     citecolor=blue,linkcolor=blue,urlcolor=teal
  ]{hyperref}

  \usepackage{tikz}
  \usetikzlibrary{arrows,arrows.meta,calc,decorations.markings,decorations.pathmorphing,backgrounds,positioning,fit}
        \tikzset{->-/.style={decoration={markings, mark=at position #1 with
        {\arrow{Stealth[length=7pt,width=4.67pt]}}},postaction={decorate}}}
        \tikzset{mapping/.style={decoration={markings,
  mark=at position #1 with {\arrow{Classical TikZ Rightarrow[length=3pt]}}},postaction={decorate}}}

  \makeatletter
  \newcommand*\if@single[3]{%
    \setbox0\hbox{${\mathaccent"0362{#1}}^H$}%
    \setbox2\hbox{${\mathaccent"0362{\kern0pt#1}}^H$}%
    \ifdim\ht0=\ht2 #3\else #2\fi
    }
  \newcommand*\rel@kern[1]{\kern#1\dimexpr\macc@kerna}
  \newcommand*\widebar[1]{\@ifnextchar^{{\wide@bar{#1}{0}}}{\wide@bar{#1}{1}}}
  \newcommand*\wide@bar[2]{\if@single{#1}{\wide@bar@{#1}{#2}{1}}{\wide@bar@{#1}{#2}{2}}}
  \newcommand*\wide@bar@[3]{%
    \begingroup
    \def\mathaccent##1##2{%
      \if#32 \let\macc@nucleus\first@char \fi
      \setbox\z@\hbox{$\macc@style{\macc@nucleus}_{}$}%
      \setbox\tw@\hbox{$\macc@style{\macc@nucleus}{}_{}$}%
      \dimen@\wd\tw@
      \advance\dimen@-\wd\z@
      \divide\dimen@ 3
      \@tempdima\wd\tw@
      \advance\@tempdima-\scriptspace
      \divide\@tempdima 10
      \advance\dimen@-\@tempdima
      \ifdim\dimen@>\z@ \dimen@0pt\fi
      \rel@kern{0.6}\kern-\dimen@
      \if#31
        \overline{\rel@kern{-0.9}\kern\dimen@\macc@nucleus\rel@kern{0.4}\kern\dimen@}%
        \advance\dimen@0.4\dimexpr\macc@kerna
        \let\final@kern#2%
        \ifdim\dimen@<\z@ \let\final@kern1\fi
        \if\final@kern1 \kern-\dimen@\fi
      \else
        \overline{\rel@kern{-0.9}\kern\dimen@#1}%
      \fi
    }%
    \macc@depth\@ne
    \let\math@bgroup\@empty \let\math@egroup\macc@set@skewchar
    \mathsurround\z@ \frozen@everymath{\mathgroup\macc@group\relax}%
    \macc@set@skewchar\relax
    \let\mathaccentV\macc@nested@a
    \if#31
      \macc@nested@a\relax111{#1}%
    \else
      \def\gobble@till@marker##1\endmarker{}%
      \futurelet\first@char\gobble@till@marker#1\endmarker
      \ifcat\noexpand\first@char A\else
        \def\first@char{}%
      \fi
      \macc@nested@a\relax111{\first@char}%
    \fi
    \endgroup
  }
  \makeatother

  \begingroup
    \makeatletter
    \@for\theoremstyle:=definition,remark,plain\do{%
      \expandafter\g@addto@macro\csname th@\theoremstyle\endcsname{%
        \addtolength\thm@preskip\parskip
        }%
      }
  \endgroup

  \newcommand\address[1]{}
  \newcommand\email[1]{}
  \newcommand\dedicatory[1]{}

  \theoremstyle{plain}
  \newtheorem{theorem}{Theorem}[section]
  \newtheorem{proposition}[theorem]{Proposition}
  \newtheorem{corollary}[theorem]{Corollary}
  \newtheorem{lemma}[theorem]{Lemma}

	\newtheorem*{keylem}{Theorem 6.6}

  \theoremstyle{definition}
  \newtheorem{definition}[theorem]{Definition}
  \newtheorem{remark}[theorem]{Remark}

  \newtheorem*{claim*}{Claim}

  \newtheorem*{question*}{Question}
  \newtheorem*{answer*}{Answer}
  \newtheorem*{application*}{Application}
  \newtheorem*{cgbthm}{Combinatorial Gauss--Bonnet Theorem}

  \newcommand{\secref}[1]{Section~\ref{Sec:#1}}
  \newcommand{\thmref}[1]{Theorem~\ref{Thm:#1}}
  \newcommand{\corref}[1]{Corollary~\ref{Cor:#1}}
  \newcommand{\lemref}[1]{Lemma~\ref{Lem:#1}}
  \newcommand{\propref}[1]{Proposition~\ref{Prop:#1}}
  
  \newcommand{\remref}[1]{Remark~\ref{Rem:#1}}

  \newcommand{\Z}{\ensuremath{\mathbf{Z}}\xspace}
  \newcommand{\R}{\ensuremath{\mathbf{R}}\xspace}

  \newcommand{\Link}{\ensuremath{\operatorname{Link}}\xspace}
  \newcommand{\Corners}{\ensuremath{\operatorname{Corners}}\xspace}

  \renewcommand{\H}{\mathcal H}
  \renewcommand{\k}{\kappa}
  \newcommand{\dS}{\partial S}
  
  \renewcommand{\epsilon}{\varepsilon}

  \newcommand{\ov}{\widebar}

  \newcommand{\fat}{\ensuremath{X_{\epsilon}}\xspace}	

  \DeclareMathOperator{\scl}{scl}
  \DeclareMathOperator{\cl}{cl}

  \DeclareMathOperator{\genus}{genus}

  \newcommand{\set}[1]{\ensuremath{\left\{\, {#1} \,\right\}}\xspace} 
  \newcommand{\abs}[1]{\ensuremath{\left\lvert {#1} \right\rvert}\xspace} 
	\newcommand{\card}[1]{\ensuremath{\#{#1}}\xspace}

  \newcommand{\st}{\ensuremath{\,\, \colon \,\,}\xspace} 
  \newcommand{\from}{\ensuremath{\colon \thinspace}\xspace} 
  \newcommand{\into}{\ensuremath{\hookrightarrow}\xspace} 

  \newcommand{\half}{\ensuremath{\mathcal{H}}\xspace}
  \newcommand{\trans}{\ensuremath{\pitchfork}\xspace}

  \AtEndDocument{\bigskip{\footnotesize%
    Max Forester: \texttt{mf@ou.edu}%
    \vspace{.7em} \\
    Ignat Soroko: \texttt{ignat.soroko@ou.edu}%
    \vspace{.7em} \\
    Jing Tao: \texttt{jing@ou.edu}%
  }} 

  \fancypagestyle{plain}

  \newcommand\blfootnote[1]{%
    \begingroup%
    \renewcommand\thefootnote{}\footnote{#1}%
    \addtocounter{footnote}{-1}%
    \endgroup%
  }

\begin{document}


  \title {Genus bounds in right-angled Artin groups}
  \author {Max Forester, Ignat Soroko, and Jing Tao}
  \date{}

  \maketitle

  \begin{abstract}
      
    \noindent 
    We show that in any right-angled Artin group whose defining graph
    has chromatic number $k$, every non-trivial element has stable
    commutator length at least $1/(6k)$. Secondly, if the defining
    graph does not contain triangles, then every non-trivial element has
    stable commutator length at least $1/20$. 
    These results are obtained via an elementary geometric argument
    based on earlier work of Culler. 
    \blfootnote{2010 \emph{Mathematics Subject Classification.}
    57M07, 20F65, 20F67.} \blfootnote{\emph{Date:} \date{\today}}

  \end{abstract}

  \thispagestyle{empty}
 

  %
  %
 
\section{Introduction}

  In a topological space $X$ with fundamental group $G$, a loop $\gamma
  \from S^1 \to X$ representing an element $g\in G$ may extend to a map of
  an oriented surface $S \to X$ with boundary $\gamma$. The smallest genus
  of such a surface is called the \emph{commutator length} (cl) of $g \in G$.
  The \emph{stable commutator length} (scl) of $g$ is defined to be the
  limit $\scl(g) = \lim_{n \to \infty} \cl(g^n)/n$. These quantities have
  relevance in several areas, particularly low-dimensional topology,
  bounded cohomology, and dynamics (see~\cite{Cal} and references therein). 

  Both commutator length and stable commutator length can be very difficult
  to compute. Understanding the qualitative behavior of scl is a somewhat
  more tractable problem. For many important classes of groups it has been
  shown that the spectrum of values of scl has a gap above zero (e.g.
  \cite{CF,CFL,BBF}). An early result along these lines is due to Culler.
  In~\cite{Cul} he gave a lower bound for the stable commutator length of
  elements in a free group $F$: for every non-trivial $g\in F$,   
  \[
  \scl(g) \ \ge \ \frac16.
  \]

  The purpose of this note is to generalize Culler's argument to the case of
 right-angled Artin groups (RAAGs) in two different ways. We obtain: 

  \begin{theorem}\label{Thm:1}
    Let $G=A(\Gamma)$ be a right-angled Artin group whose defining
    graph $\Gamma$ has chromatic number $k$. Then every 
    non-trivial element $g\in G$ satisfies: 
    \[ \scl(g) \ \ge \ \frac1{6k}. \]
  \end{theorem}

  \begin{theorem}\label{Thm:2}
    Let $G=A(\Gamma)$ be a right-angled Artin group whose defining
    graph $\Gamma$ does not contain triangles. Then every 
    non-trivial element $g\in G$ satisfies: 
    \[ \scl(g) \ \ge \ \frac1{20}. \]
  \end{theorem}

  Note that \thmref{2} is not a consequence of \thmref{1}, as
  demonstrated by the existence of triangle-free graphs with large chromatic
  number, such as Mycielski's graphs~\cite{Myc}. 

  It should be noted that Culler's result has since been improved: Duncan
  and Howie \cite{DH} showed that $\scl(g) \geq 1/2$ for all non-trivial $g
  \in F$ (see \cite{Chen} for another proof of this result). This is
  the best possible lower bound for free groups since $\scl([a,b]) =
  1/2$ in $\langle a,b\rangle$. 

  We also note that Heuer, very recently, obtained the same lower
  bound of $1/2$ for scl in any right-angled Artin group
  \cite{Heuer}. His method is based on constructing quasimorphisms, as
  was the previous general lower bound of $1/24$ established in
  \cite{FFT}. 
  The arguments presented here are quite elementary and geometric in
  nature, and we believe that they are of independent interest.

  \subsection*{Methods}

  The proofs of Theorems \ref{Thm:1} and \ref{Thm:2} find explicit
  lower bounds for $\cl(g^n)$ in terms of $n$, by considering a map of
  a surface $S$ into a Salvetti complex $X(\Gamma)$, with boundary
  representing $g^n$. Taking pre-images of the hyperplanes, we obtain
  one-dimensional submanifolds of $S$, transversely labeled by
  generators of the right-angled Artin group. Unlike in the free
  groups setting, these curves may cross each other, if their labels
  are generators that commute. 

  \thmref{1} is obtained by showing that the collection of curves
  includes a suitably large sub-collection of properly embedded arcs
  that are pairwise disjoint and non-parallel in $S$. 
	
  For~\thmref{2}, we first ``tighten'' the pattern of
  curves and reduce the 
  genus of $S$; the resulting pattern of curves cuts $S$ into polygons, all
  having four or more sides. This is where the two-dimensionality
  assumption is used (otherwise there could be triangles). This
  combinatorial structure now has non-positive curvature and we can
  estimate $\chi(S)$ using a Gauss--Bonnet formula. 

  An easy but crucial step in Culler's argument is to observe that a
  cyclically reduced word $w$ in a free group cannot contain overlapping
  subwords of the form $u$ and $u^{-1}$. Thus, if $w^n$ contains both $u$
  and $u^{-1}$ then $\abs{u} \leq \abs{w}/2$. This fact is used to relate
  the amount of negative curvature in $S$ to the exponent $n$. 

  In the case of a right-angled Artin group $G$, we must consider pairs of
  subwords $u$, $u^{-1}$ appearing in a cyclically reduced word $w$ that
  \emph{represents} $g^n \in G$;  we need to know that such words always
  satisfy $\abs{u} \leq \abs{w}/(2n)$. The difficulty is that $G$ has
  relations and $w$ itself need not be a proper power. We call this
  property the \emph{non-overlapping property} (by analogy with the free
  groups case) and prove it for all right-angled Artin groups in
  \thmref{NoOverlapping}. This result is needed in both \thmref{1} and
  \thmref{2}. 

  The non-overlapping property for right-angled Artin groups appears to be
  rather non-trivial. For instance, the proof makes use of all of Haglund
  and Wise's axioms for special cube complexes from \cite{HaglundWise}. In
  \secref{NoOv} we give the proof, after reviewing some of the terminology
  and ideas from \cite{FFT} concerning essential characteristic sets
  in CAT(0) cube complexes. 

\subsection*{Acknowledgments}
    This material is based upon work supported by NSF grant DMS-1440140
    while the second and third authors were in residence at the
    Mathematical Sciences Research Institute in Berkeley, California,
    during the Fall 2016 semester. The second author acknowledges partial
    support from NSF grants DMS-1107452, DMS-1107263, DMS-1107367 ``RNMS:
    Geometric Structures and Representation Varieties'' (the GEAR Network).
    The second and third authors also acknowledge partial support from NSF
    grants DMS-1611758 and DMS-1651963. 

\section{Preliminaries}

  \label{Sec:Prelim}

  We remind the reader of some relevant definitions. 

  \begin{definition}[Commutator length, stable commutator length] \label{Def:scl} 
    Given $G$, let $X$ be a path connected space with fundamental group
    $G$. For any $g \in [G,G]$ the \emph{commutator length} of $g$ is
    equal to 
    \[ \cl(g) \ = \ \min_S \genus(S)\]
    where the minimum is taken over all continuous maps of surfaces $f
    \from S \to X$ such that $S$ is compact, connected, oriented, has one
    boundary component, and the restriction $f \vert_{\partial S}
    \from \partial S \to X$ represents the conjugacy class of $g$
    in $\pi_1(X)$. Such a surface will be called (in this paper) an
    \emph{admissible surface} for $g$. The condition that $g \in [G,G]$
    ensures that admissible surfaces exist. 

    The \emph{stable commutator length} of $g \in [G,G]$ is defined by
    the convergent limit 
    \[ \scl(g) \ = \ \lim_{n \to \infty} \cl(g^n)/n. \]
    See \cite{Cal} for details on convergence, and for other basic
    properties and equivalent definitions. 
    If $g^n \in [G,G]$ for some $n \not= 0$ we may define $\scl(g) =
    \scl(g^n)/n$, which is consistent with the first definition because of
    the identity $\scl(g^n) = n  \scl(g)$. If $g^n \not\in [G,G]$ for any
    $n \not= 0$ then $\scl(g) = \infty$ by convention. 
  \end{definition}

  \begin{definition}[Right-angled Artin group]
    Let $\Gamma$ be a finite simplicial graph with vertex set
    $V(\Gamma)$ and edge set $E(\Gamma)$. The \emph{right-angled
    Artin group}, or RAAG, associated to $\Gamma$ is a finitely presented
    group $A(\Gamma)$ given by the presentation 
    \[
    A(\Gamma) \ = \ \left\langle\, a\in V(\Gamma)  \mid [a,b]=1 \text{ if } (a,b)\in
      E(\Gamma)\,\right\rangle. 
    \] 
  \end{definition}

  \begin{definition}[Salvetti complex]
		For each $a\in V(\Gamma)$ let $S^1_{a}$ be a circle endowed
    with the structure of a CW complex having a single $0$--cell and a single
    $1$--cell. Let $n=\#V(\Gamma)$ be the number of vertices of $\Gamma$
    and let $T=\prod_{a\in V(\Gamma)} S^1_{a}$ be an $n$--dimensional torus
    with the product CW structure. For every complete subgraph
    $K\subseteq\Gamma$ with $V(K)=\{a_{1},\dotsc,a_{k}\}$ define a
    $k$--dimensional torus $T_K$ as the Cartesian product of CW
    complexes $T_K=\prod_{i=1}^k S^1_{a_{i}}$ and observe that $T_K$ can be
    identified as a combinatorial subcomplex of $T$. Then the
    \emph{Salvetti complex associated with $A(\Gamma)$} is  
    \[
    X(\Gamma) \ = \ \bigcup\set{T_K\subseteq T\mid K \text{ a complete subgraph of
      }\Gamma}. 
    \]
    Thus $X(\Gamma)$ has a single $0$--cell and $n$ $1$--cells. Each edge
    $(a,b)\in E(\Gamma)$ contributes a square $2$--cell to $X(\Gamma)$ with
    the attaching map $aba^{-1}b^{-1}$. In general each complete
    subgraph $K\subseteq \Gamma$ contributes a $k$--dimensional cell to
    $X(\Gamma)$ where $k=\#V(K)$. 
  \end{definition}

  \begin{definition}[Fat Salvetti complex]
    Given $\epsilon > 0$, define the \emph{fat Salvetti complex} associated
    to $A(\Gamma)$ to be the open $\epsilon$--neighborhood of $X(\Gamma)$
    in the torus $T$, denoted $\fat(\Gamma)$. There is a deformation
    retraction of $\fat(\Gamma)$ onto $X(\Gamma)$ for $\epsilon$
    sufficiently small. Fix such an $\epsilon$ from now on. 
  \end{definition}

  The fat Salvetti complex is convenient for carrying out an easy
  transversality argument. One could alternatively use an approach to
  transversality similar to those in \cite{Rourke, Stallings, BF}. 

  Following~\cite{McWise} we now introduce a useful tool for computing the
  Euler characteristic of a $2$--di\-men\-si\-o\-nal complex.

  \begin{definition}[Corners, Angled $2$--complex]
    Let $X$ be a locally finite combinatorial $2$--complex and $v$ a
    $0$--cell of $X$. We will refer to the edges of $\Link(v)$ as
    \emph{corners} of $v$. $X$ is called an \emph{angled $2$--complex} if
    it has an angle $\angle c\in\R$ associated to each corner $c$ of every
    $0$--cell of $X$. 
  \end{definition}

  \begin{definition}[Curvature]
    For every $0$--cell $v$ of $X$, its \emph{curvature} $\k(v)$ is defined
    as 
    \[
    \k(v)=2\pi-\pi\chi(\Link(v))-\sum_{c\in\Corners(v)}\angle c.
    \]
    For every $2$--cell $f$ of $X$, its \emph{curvature} $\k(f)$ is defined
    as 
    \[
    \k(f)=\sum_{c\in\Corners(f)}\angle c- (P(f)-2)\pi,
    \]
    where 
    \[
    \Corners(f) =\{ \text{ edges in $\Link(v)$ contained in $f$ } \mid
    \text{ $v$ is a $0$--cell belonging to $f$ }\}. 
    \]
    and $P(f)$ is the combinatorial length of the boundary of $f$.
  \end{definition}

  The following theorem was proved in~\cite[Theorem 4.6]{McWise}:
  \begin{cgbthm}
  Let $X$ be a compact angled $2$--complex. Then 
  \[\pushQED{\qed}
  2\pi\chi(X)=\sum_{v\in \text{$0$--cells
  }}\k(v)+\sum_{f\in \text{$2$--cells
  }}\k(f).\qedhere\popQED
  \]
  \end{cgbthm}

\section{Surfaces with patterns}

  \label{Sec:Patterns}

  \subsection*{Mapping a surface to a fat Salvetti complex}

  From now on we let $\Gamma$ be a finite simplicial graph. For each
  $a \in V(\Gamma)$ let $e_a$ be the corresponding oriented edge of
  $X(\Gamma)$ (so that $e_a$, considered as a based loop in $X(\Gamma)
  \subset \fat(\Gamma)$, represents the element $a$ of
  $\pi_1(\fat(\Gamma)) = A(\Gamma)$). 

  Recall that the subcomplex of $X(\Gamma)$ determined by $e_a$ is a circle
  $S^1_a$, and note that there is a retraction $r_a \from \fat(\Gamma) \to
  S^1_a$ which is the restriction of the projection map from the torus $T$
  to the factor $S^1_a$. We define the \emph{hyperplane dual to $a$} in
  $\fat(\Gamma)$ to be the pre-image of the midpoint of $e_a$ under this
  retraction. It is denoted $H_a$. Note that this is not quite the usual
  definition of hyperplane, because we are working in the fat Salvetti
  complex $\fat(\Gamma)$. In fact $\fat(\Gamma)$ is a manifold and the
  hyperplanes $H_a$ are codimension-one submanifolds. 

  Let $f \from S \to \fat(\Gamma)$ be an admissible surface for $g \in
  A(\Gamma)$. The composition $S \to \fat(\Gamma) \into T$ is the product
  of the component maps $f_a = r_a \circ f$, since $T = \prod_{a} S^1_a$.
  Each of these maps $f_a \from S \to S^1_a$ may be changed by an
  arbitrarily small homotopy to arrange that the midpoint of $e_a$ is a
  regular value of $f_a$. Since $\fat(\Gamma)$ is open in $T$, this can be
  achieved for all $a$ by homotopies such that the product homotopy is a
  homotopy of $f$ \emph{inside $\fat(\Gamma)$}. Thus, after such a homotopy
  of $f$, we may assume that $f$ is simultaneously transverse to all of the
  hyperplanes $H_a$. Then, for each $a$, the pre-image $M_a = f^{-1}(H_a)$
  is a compact properly embedded one-dimensional submanifold of $S$. The
  submanifolds $M_a$ and $M_b$ may intersect, but they will only do so
  transversely, in the interior of $S$, and only if $(a,b)$ is an edge of
  $\Gamma$. If $(a,b)$ is not an edge, then $M_a$ and $M_b$ will be
  disjoint because $H_a \cap H_b = \emptyset$. 

  Each component of $M_a$ comes with a transverse orientation (i.e.\ a
  choice of normal direction) which we label by the generator $a$. This is
  the transverse orientation induced by the orientation of $e_a$ under the
  map $r_a \circ f$. The opposite transverse orientation is labeled by
  $a^{-1}$. Any small arc $\alpha$ in $S$ that crosses $M_a$ in one point
  will map by $f$ to an arc that crosses $H_a$ in one point. The direction
  that it crosses in will agree with $e_a$ if and only if $\alpha$ is
  oriented with the transverse orientation of $M_a$. 

  More generally, any path $\alpha$ in $S$ which crosses the submanifolds
  $M_a$ transversely in distinct points has a corresponding word
  $w_{\alpha}$ in the generators of $A(\Gamma)$ and their inverses; the
  letters of $w_{\alpha}$ are the labels assigned to the transverse
  orientations followed by $\alpha$ as it passes through each crossing. If
  $\alpha$ is a based loop with basepoint $p$ disjoint from the
  submanifolds $M_a$, then the word $w_{\alpha}$ represents the element $[f
  \circ \alpha]$ in $\pi_1(\fat(\Gamma),f(p))$.

  In particular, the oriented boundary of $S$ crosses the endpoints of
  the manifolds $M_a$ transversely, and has an associated cyclic word
  $w_{\partial S}$. This word represents the conjugacy class of $g$ in
  $\pi_1(\fat(\Gamma))$. 

  \subsection*{Simplification\label{Sec:simp}}

  We now have a compact surface $S$ together with what we call a
  \emph{pattern on $S$}: for each $a \in V(\Gamma)$, a properly embedded
  submanifold $M_a$ of $S$ with a choice of transverse orientation (labeled
  $a$) on each component. The surface with pattern $(S, \{M_a\}_{a \in
  V(\Gamma)})$ satisfies:  

  \begin{enumerate}[label=T\arabic*:, ref=T\arabic*]
    \item \label{t1} if $M_a$ and $M_b$ intersect, then $(a,b)$ is an edge
      of $\Gamma$, and the intersections are transverse and occur in the
      interior of $S$; 
    \item \label{t2} the cyclic word $w_{\partial S}$ given by the
      transverse labels along $\partial S$ represents the conjugacy class
      of $g$ in $A(\Gamma)$.  
  \end{enumerate}

  At this point we have no further need for the continuous map $f \from S
  \to \fat(\Gamma)$. We will simplify both $S$ and the pattern $\{M_a\}_{a
  \in V(\Gamma)}$ by applying some moves. These moves will preserve
  properties \ref{t1} and \ref{t2}. These two properties, along with the
  additional properties achieved by the moves, are all that will be
  needed to estimate $\chi(S)$. In order to describe the moves we need
  some further terminology. 

  \begin{definition} \label{Def:polygon}
    A \emph{pattern curve} is a connected component of $M_a$ for any
    $a \in V(\Gamma)$. A pattern curve is called a
    \emph{pattern arc} if it is homeomorphic to an interval, and a
    \emph{pattern loop} if it is homeomorphic to a circle. 

    An \emph{intersection point} is a point of intersection either between
    two pattern curves, or between a pattern curve and $\partial S$. 

    Now let $\mathcal{M}$ be any union of pattern curves. If we cut
    $S - \partial S$ along $\mathcal{M}$, we get pieces, some of which
    may be open disks. Each such disk has a characteristic map $D^2 \to S$ 
    giving $D^2$ the structure of a polygon with some number of sides.
    Here, a $\emph{side}$ is a maximal connected subset of $\partial D^2$
    mapping into a single pattern curve of $\mathcal{M}$ or into $\partial
    S$. Informally, we say that the original open disk is a polygon
    with that number of sides (even though distinct sides of $D^2$ may
    map to the same arc in $S$). 
  \end{definition}

  \begin{definition}

    A \emph{bigon} is a complementary component of $\mathcal{M}$ in $S
    - \partial S$ which is a polygon with two sides (for some
    $\mathcal{M}$, which may be taken to be just one or two pattern
    curves). A bigon is called \emph{innermost} if every pattern curve
    that intersects one side also intersects the other side.  

    A \emph{half-bigon} is a complementary component of $\mathcal{M}$
    which is a polygon with three sides, one in $\partial S$. Call the
    other two sides \emph{interior} sides. A half-bigon is called
    \emph{innermost} if every pattern curve that intersects one
    interior side also intersects the other interior side. 

  \end{definition}

  \begin{remark} \label{innermost}

    If a bigon $B$ is not innermost, then it properly contains
    a smaller bigon. Similarly, if a half-bigon $B$ is not innermost, then it
    properly contains either a half-bigon or a bigon. 
    It follows that if there are no innermost bigons or half-bigons,
    then there are no bigons or half-bigons at all. 

  \end{remark} 

  The moves are as follows: 

  \begin{enumerate}
    \item \label{m2} Remove an innermost bigon (whose sides are not
      in $\partial S$). See Figure \ref{Fig:bigon}. 

      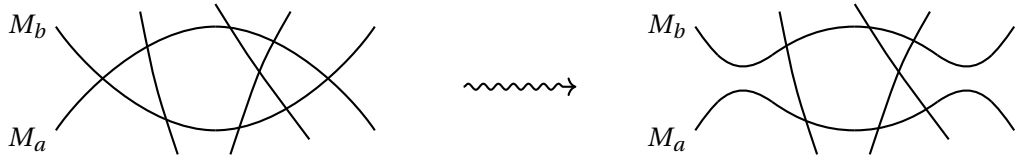
\begin{figure}[htbp!]
      \begin{center}	
      \begin{tikzpicture}[scale=1.0,thick]
				\draw (0.88,2.12) .. controls (1.2,2.6) and (2.1,3.5) .. (3,3.5);
				\draw (3,3.5) .. controls (3.9,3.5) and (4.8,2.6) .. (5.12,2.12);
				\draw (0.88,3.5) .. controls (1.2,3.02) and (2.1,2.12) .. (3,2.12);
				\draw (3,2.12) .. controls (3.9,2.12) and (4.8,3.02) .. (5.12,3.5);
				\draw (2,3.7) .. controls (2.2,2.7) .. (2.5,1.8);
				\draw (3,3.8) .. controls (3.5,3) .. (4.25,2);
				\draw (4,3.7) .. controls (3.6,3) .. (3.2,1.8);
        \draw (0.5,3.5) node {$M_b$};
        \draw (0.5,2.0) node {$M_a$};
				\draw [->,decorate,decoration={snake,amplitude=.4mm,segment length=2mm,post length=1mm}] (6.3,2.7)--(7.8,2.7);
				\draw (9.38,2.12) .. controls (9.54,2.35) .. (9.68,2.5);
				\draw (10.4,2.5) .. controls (10.75,2.25) and (11.1,2.12) .. (11.5,2.12);
				\draw (9.68,2.5) .. controls (9.9,2.7) and (10.11,2.7) .. (10.4,2.5);
				\draw (13.62,2.12) .. controls (13.46,2.35) .. (13.32,2.5);
				\draw (12.6,2.5) .. controls (12.25,2.25) and (11.9,2.12) .. (11.5,2.12);
				\draw (13.32,2.5) .. controls (13.1,2.7) and (12.89,2.7) .. (12.6,2.5);
				\draw (9.38,3.5) .. controls (9.54,3.27) .. (9.68,3.12);
				\draw (10.4,3.12) .. controls (10.75,3.37) and (11.1,3.5) .. (11.5,3.5);
				\draw (9.68,3.12) .. controls (9.9,2.92) and (10.11,2.92) .. (10.4,3.12);
				\draw (13.62,3.5) .. controls (13.46,3.27) .. (13.32,3.12);
				\draw (12.6,3.12) .. controls (12.25,3.37) and (11.9,3.5) .. (11.5,3.5);
				\draw (13.32,3.12) .. controls (13.1,2.92) and (12.89,2.92) .. (12.6,3.12);				
				\draw (10.5,3.7) .. controls (10.7,2.7) .. (11,1.8);
				\draw (11.5,3.8) .. controls (12,3) .. (12.75,2);
				\draw (12.5,3.7) .. controls (12.1,3) .. (11.7,1.8);
				\draw (9.0,3.5) node {$M_b$};
        \draw (9.0,2.0) node {$M_a$};
      \end{tikzpicture}
      \caption{\label{Fig:bigon}Removing an innermost bigon (Move \ref{m2})}
      \end{center}
      \end{figure}
          
    \item \label{m3} Remove an innermost half-bigon. See Figure
      \ref{Fig:half-bigon}. 
 
    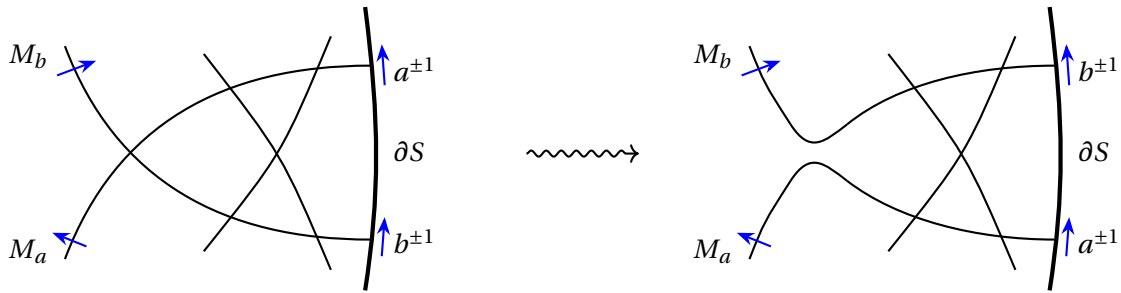
\begin{figure}[htbp!]
    \begin{center}	
    \begin{tikzpicture}[scale=1.3,thick]
      \draw (0.88,1.62) .. controls (1.4,3) and (2.5,3.6) .. (4,3.6);
      \draw (0.88,3.8) .. controls (1.4,2.42) and (2.5,1.82) .. (4,1.82);
      \draw [ultra thick] (3.95,1.3) .. controls (4.1,2.27) and (4.1,3.03) .. (3.95,4.2);
			
			\draw (0.5,3.7) node {$M_b$};
      \draw (0.5,1.7) node {$M_a$};	
      \draw (4.4,2.72) node {$\dS$};
      \draw [blue, arrows={-Stealth}] (0.8,3.5)--(1.2,3.65);
      \draw [blue, arrows={-Stealth}] (1.1,1.75)--(0.75,1.89);
      \draw [blue, arrows={-Stealth}] (4.15,3.4)--(4.12,3.82); \draw (4.45,3.56) node {$a^{\pm1}$};
      \draw [blue, arrows={-Stealth}] (4.12,1.65)--(4.15,2.05); \draw (4.45,1.8) node {$b^{\pm1}$};
			
			\draw (2.3,1.7) .. controls (3.1,2.7) .. (3.6,3.9);
			\draw (2.3,3.72) .. controls (3.1,2.69) .. (3.6,1.51);
						
      \draw [->,decorate,decoration={snake,amplitude=.4mm,segment length=2mm,post length=1mm}] (5.6,2.7)--(6.75,2.7);
      
      \draw [ultra thick] (10.95,1.3) .. controls (11.1,2.27) and (11.1,3.03) .. (10.95,4.2);
			\draw (7.88,1.62) .. controls (8.03,2) .. (8.3,2.42);
			\draw (8.88,2.42) .. controls (9.4,2.02) and (10.2,1.82) .. (11,1.82);
			\draw (8.3,2.42) .. controls (8.5,2.73) and (8.646,2.6) .. (8.88,2.42);
 			\draw (7.88,3.8) .. controls (8.03,3.42) .. (8.3,3);
			\draw (8.88,3) .. controls (9.4,3.42) and (10.2,3.6) .. (11,3.6);
			\draw (8.3,3) .. controls (8.5,2.69) and (8.646,2.82) .. (8.88,3);
			
			\draw (7.5,3.7) node {$M_b$};
      \draw (7.5,1.7) node {$M_a$};	
      \draw (11.4,2.72) node {$\dS$};
      \draw [blue, arrows={-Stealth}] (7.8,3.5)--(8.2,3.65);
      \draw [blue, arrows={-Stealth}] (8.1,1.75)--(7.75,1.89);
      \draw [blue, arrows={-Stealth}] (11.15,3.4)--(11.12,3.82); \draw (11.45,3.56) node {$b^{\pm1}$};
      \draw [blue, arrows={-Stealth}] (11.12,1.65)--(11.15,2.05); \draw (11.45,1.8) node {$a^{\pm1}$};
			
			\draw (9.3,1.7) .. controls (10.1,2.7) .. (10.6,3.9);
			\draw (9.3,3.72) .. controls (10.1,2.69) .. (10.6,1.51);

    \end{tikzpicture}	
    \caption{\label{Fig:half-bigon}Removing an innermost half-bigon (Move \ref{m3})}
    \end{center}
    \end{figure}
          
    \item \label{m4.5} Splice together two endpoints of $M_a$ that land
      on adjacent cancellable letters of $w_{\partial S}$. See Figure
      \ref{Fig:splice}. 

    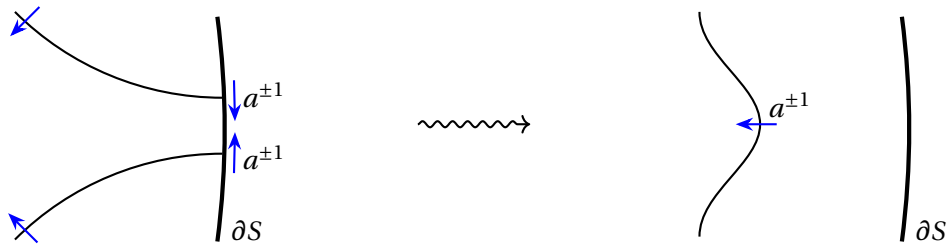
\begin{figure}[htbp!]
    \begin{center}
    \begin{tikzpicture}[scale=1.3,thick]
      \draw (0.88,1.52) arc [start angle=135, end angle=90, x radius=3, y radius=3];
      \draw (0.88,3.85) arc [start angle=225, end angle=270, x radius=3, y radius=3];
      \draw [ultra thick] (2.95,1.5) .. controls (3.05,2.27) and (3.05,3.03) .. (2.95,3.8);
      \draw (3.25,1.62) node {$\dS$};
      \draw [blue, arrows={-Stealth}] (1.13,3.9)--(0.83,3.6);
      \draw [blue, arrows={-Stealth}] (1.11,1.49)--(0.81,1.79);
      \draw [blue, arrows={-Stealth}] (3.12,2.2)--(3.13,2.62); \draw (3.42,2.36) node {$a^{\pm1}$};
      \draw [blue, arrows={-Stealth}] (3.12,3.15)--(3.13,2.73); \draw (3.42,3.0) node {$a^{\pm1}$};
      \draw [->,decorate,decoration={snake,amplitude=.4mm,segment length=2mm,post length=1mm}] (5.0,2.7)--(6.15,2.7);
      \draw (7.88,3.85) .. controls (7.88,3.4) and (8.5,3.1) .. (8.5,2.7)
                                                                              .. controls (8.5,2.3) and (7.88,2.0) .. (7.88,1.55);
      \draw [blue, arrows={-Stealth}] (8.67,2.7)--(8.25,2.7);
      \draw (8.8,2.9) node {$a^{\pm1}$};
      \draw [ultra thick] (9.95,1.5) .. controls (10.05,2.27) and (10.05,3.03) .. (9.95,3.8);
      \draw (10.25,1.62) node {$\dS$};
    \end{tikzpicture}
    \caption{\label{Fig:splice}Splicing adjacent endpoints (Move \ref{m4.5})}
    \end{center}	
    \end{figure}		
            
    \item \label{m1} Discard any pattern loop. 

    \item \label{m5} Perform surgery along a non-separating simple closed
      curve in $S$ that is disjoint from every pattern curve; that is, replace an
      annular neighborhood (also chosen disjoint from every pattern curve) with
      two disks. 
  \end{enumerate}

  \begin{remark}
    The only move that changes the topology of $S$ is move \ref{m5}. This
    move increases $\chi(S)$ by $2$, but preserves the fact that $S$ is
    connected and has one boundary component. Such a surface has Euler
    characteristic at most $1$.
  \end{remark}

  \begin{remark}
    The only moves that change $w_{\partial S}$ are moves \ref{m3} and
    \ref{m4.5}. Move \ref{m3} exchanges two adjacent letters of
    $w_{\partial S}$ which are commuting elements of $A(\Gamma)$. Move
    \ref{m4.5} removes a subword $a a^{-1}$ or $a^{-1} a$ from $w_{\partial
    S}$. In both cases, $w_{\partial S}$ still represents $g$ after the
    move (that is, property \ref{t2} is preserved by all moves). 
  \end{remark}

  \begin{remark}

    In moves \ref{m2} and \ref{m3}, the new pattern curves
    after the move cross exactly the same pattern curves that they did
    before the move, by the innermost property of the bigon or
    half-bigon. Thus property \ref{t1} is preserved by all moves (the
    other cases being obvious). 

  \end{remark}

  Now define the \emph{complexity} of $(S, \{M_a\}_{a \in V(\Gamma)})$ to
  be the sum of three quantities: the total number of pattern curves,
  the total number of intersection points (including intersections
  with $\partial S$), and $1-\chi(S)$. The complexity is a
  non-negative integer. 

  \begin{lemma}
    Each of the moves \ref{m2}--\ref{m5} decreases the complexity of\/ $(S,
    \{M_a\}_{a \in V(\Gamma)})$. 
  \end{lemma}

  \begin{proof}
    Moves \ref{m2} and \ref{m3} reduce the number of intersection points
    without changing the other two components of complexity. Move
    \ref{m4.5} reduces the number of intersection points by $2$, and
    possibly also the number of pattern curves (unless it turns an 
    arc into a loop). It does not change $1 - \chi(S)$.
    Move \ref{m1} reduces the number of pattern curves, and possibly
    also the number of intersection points. It leaves $1 - \chi(S)$
    unchanged. Move \ref{m5} reduces $1-\chi(S)$ by $2$ without
    changing the other two quantities. 
  \end{proof}

  Starting with $(S, \{M_a\}_{a \in V(\Gamma)})$, one can perform moves, in
  any order, until the complexity cannot be reduced any further. Since no
  moves are available, we may conclude several things (in addition to
  properties \ref{t1}, \ref{t2}): 

  \begin{enumerate}[label=T\arabic*:, ref=T\arabic*, start=3]
    \item \label{t3} the word $w_{\partial S}$ is cyclically reduced (or
      move \ref{m4.5} could be performed);
    \item \label{t4} there are no bigons or half-bigons:

      An innermost half-bigon, or an innermost bigon not on $\partial
      S$, cannot exist (or move \ref{m2} or \ref{m3} is
      available). An innermost bigon on $\partial S$ either 
      contains a half-bigon (which contains an innermost half-bigon or
      bigon, which is a contradiction), or makes move \ref{m4.5}
      available. The claim now follows from Remark \ref{innermost}. 

    \item \label{t5} the union $\bigcup_{a \in V(\Gamma)} M_a$ cuts $S$
      into disks:

      Consider a simple closed curve in $S - \partial S$ that is
      disjoint from 
      $\mathcal{M} = \bigcup_{a \in V(\Gamma)} M_a$. If it is
      non-separating then move \ref{m5} is available. If it is
      separating then $\mathcal{M}$ lies entirely on one side, since
      every pattern curve is a pattern arc meeting $\partial S$. The
      other side either is a disk or admits move~\ref{m5}. 

    \item \label{t6} if $\Gamma$ has no triangles then each of these
      disks is a polygon with at least four sides:

      The number of sides cannot be $1$ since there are no pattern
      loops. It cannot be $2$ because there are no bigons. If a
      polygon has three sides and is not a half-bigon, then the sides
      are on $M_a$, $M_b$, $M_c$ with $a, b, c$ forming a triangle in
      $\Gamma$. 

  \end{enumerate}

  \begin{definition}
    A surface with pattern $(S, \{M_a\}_{a \in V(\Gamma)})$ satisfying
    properties \ref{t1} and \ref{t2}, of smallest complexity, is called a
    \emph{taut surface with pattern for $g$}. It will then also satisfy
    \ref{t3}--\ref{t6}. 
  \end{definition}

  We have just proved:

  \begin{proposition} \label{Prop:taut}
    Let $S_0$ be an admissible surface for $g$ in $\fat(\Gamma)$. Then
    there exists a taut surface with pattern $(S, \{M_a\}_{a \in
    V(\Gamma)})$ for $g$ such that $1 - \chi(S_0) \geq 1 -
    \chi(S)$. \qed 
  \end{proposition}

  \subsection*{Bands}

  Let $(S, \{M_a\}_{a\in V(\Gamma)})$ be a taut surface with pattern
  for $g$. We organize the pattern arcs into regions called bands. 

  \begin{definition} (Band)
    A \emph{rectangle} is an embedded disk in $S$ whose boundary is
    decomposed into four sides, such that two opposite sides are contained
    in $\partial S$, and the other two sides are pattern arcs. We also
    allow a rectangle to be degenerate, consisting of a single pattern arc.
    The sides in $\partial S$ are called the \emph{boundary sides} and the
    other two sides are called the \emph{interior sides} of the rectangle.
    \emph{A band} is a rectangle in $S$ which is maximal with respect to
    inclusion. Since degenerate rectangles are allowed, every pattern arc
    is contained in a band. 
  \end{definition}

  \begin{remark}\label{Rem:arc-equiv}

    Suppose $C_1$ and $C_2$ are pattern arcs that are topologically
    parallel, meaning that they are homotopic as maps of pairs
    $(I, \partial I) \to (S, \partial S)$. Then they must be disjoint
    and form part of the boundary of a rectangle $R$, since otherwise
    there would be a half-bigon. 

    If $C_3$ is a third pattern arc that is topologically parallel to
    the other two, then either it lies in $R$ between $C_1$ and $C_2$,
    or it cobounds a second rectangle $R'$ with either $C_1$ or $C_2$. In
    this case, $R \cup R'$ is a rectangle containing all three pattern
    arcs. More generally, if $C_1, \dotsc, C_k$ is a maximal family of
    pairwise parallel pattern arcs, then there is a band bounded by
    two of them, and all the others lie inside the band. Thus bands
    correspond to equivalence classes of pattern arcs under the
    relation of parallelism. 

  \end{remark}

  \begin{remark}\label{Rem:band-words}
    Let $B$ be a band, and suppose the pattern arcs inside it are
    $C_1, \dotsc, C_k$, numbered in order along $\partial S$ in one
    boundary side of the band. Let $a_1^{\epsilon_1}, \dotsc,
    a_k^{\epsilon_k}$ be their transverse labels along this side. Then
    the opposite endpoints of $C_1, \dotsc, C_k$ lie along the other
    boundary side of $B$ and their transverse labels are
    $a_k^{-\epsilon_k}, \dotsc, a_1^{-\epsilon_1}$ in order along
    $\partial S$, because $S$ is orientable (see Figure \ref{Fig:band}). Moreover, no other
    pattern arc can meet $\partial S$ in one of the boundary sides of
    $B$, since it would then form a half-bigon with some
    $C_i$. Therefore $w_{\partial S}$ contains $u = a_1^{\epsilon_1} \dotsm
    a_k^{\epsilon_k}$ and $u^{-1} = a_k^{-\epsilon_k} \dotsm
    a_1^{-\epsilon_1}$ as disjoint subwords. In fact $w_{\partial S}$
    is partitioned into these subwords, since every letter is the
    label of one end of a pattern arc, and the set of arcs is
    partitioned by the bands. 
  \end{remark}

  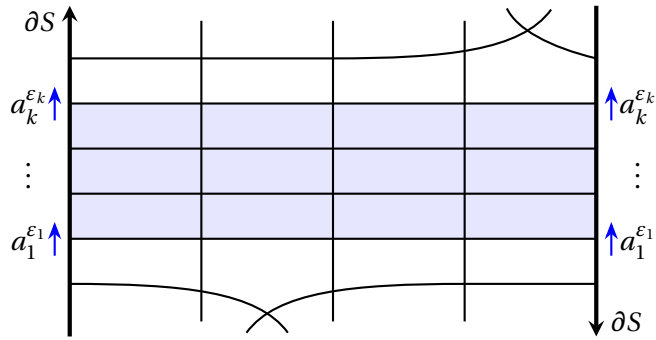
\begin{figure}[htbp!]
  \begin{center}
  \begin{tikzpicture}[scale=1.0,>=stealth,thick]
    \begin{scope}[color=blue!10]
      \fill (1,3.4)--(8,3.4)--(8,1.6)--(1,1.6);
    \end{scope}
    \draw [->,ultra thick] (1,0.3)--(1,4.7);
    \draw [<-,ultra thick] (8,0.3)--(8,4.7);
    \draw (1,1.6)--(8,1.6);
    \draw (1,2.2)--(8,2.2);
    \draw (1,2.8)--(8,2.8);
    \draw (1,3.4)--(8,3.4);
    \draw (2.75,0.5)--(2.75,4.5);
    \draw (4.5,0.5)--(4.5,4.5);
    \draw (6.25,0.5)--(6.25,4.5);
    \draw (1,4)--(4.5,4);
    \draw (4.5,4) .. controls (6.0,4.0) and (7.0,4.1) .. (7.4,4.65);
    \draw (8,4.0) arc [start angle=237, end angle=190, x radius=2.7, y radius=1];
    \draw (1,1) .. controls (2.5,1) and (3.5,0.9) .. (3.9,0.35);
    \draw (6.25,1) .. controls (4.75,1) and (3.75,0.9) .. (3.35,0.35);
    \draw (6.25,1)--(8,1);
    \draw (0.6,4.5) node {$\dS$}; \draw (8.4,0.5) node {$\dS$};
    \draw [blue, arrows={-Stealth}] (0.8,3.175)--(0.8,3.625);
    \draw [blue, arrows={-Stealth}] (8.2,3.175)--(8.2,3.625);
    \draw [blue, arrows={-Stealth}] (0.8,1.375)--(0.8,1.825);
    \draw [blue, arrows={-Stealth}] (8.2,1.375)--(8.2,1.825);
    \draw (0.45,1.55) node {$a_1^{\epsilon_1}$}; \draw (8.55,1.55) node {$a_1^{\epsilon_1}$}; 
    \draw (0.45,3.35) node {$a_k^{\epsilon_k}$}; \draw (8.55,3.35) node {$a_k^{\epsilon_k}$}; 
    \draw (0.45,2.55) node {$\vdots$}; \draw (8.55,2.55) node {$\vdots$}; 
  \end{tikzpicture}
  \caption{\label{Fig:band}The shaded region is a band, joining two
    subwords $u, u^{-1}$ of $w_{\partial S}$. In this example $\Gamma$
    has no triangles and so every face has four or more sides. }
  \end{center}
  \end{figure}

  \begin{remark}\label{Rem:band-special-face}
    We make one last observation in the case where $\Gamma$ has no
    triangles, to be used in the proof of \thmref{2}. By property
    \ref{t6}, every polygonal face of $S$ has four or more sides. 
    Let $R$ be a rectangle and $C$ one of its interior sides. Consider the
    faces of $S$ that are outside $R$ but have a side in $C$. If all of
    these faces have four sides, then the closure of their union is a
    rectangle $R'$, and $R \cup R'$ is also a rectangle, properly
    containing $R$. Therefore, for any band $B$, each of its interior sides
    is adjacent to at least one face with 5 or more sides. See again
    Figure \ref{Fig:band}. 
  \end{remark}

  Recall now that our goal is to estimate $\scl(g)$. To do this we need to
  estimate $\cl(g^n)$ in terms of $n$. The following theorem, proved
  in~\secref{NoOv}, provides the connection to the exponent $n$. 

  \begin{keylem}
    Let $w$ be a cyclically reduced word in the generators of $A(\Gamma)$
    representing the conjugacy class of the element $g^n$ in $A(\Gamma)$,
    and suppose $u$ is a word such that both $u$ and $u^{-1}$ appear as
    subwords of $w$ (considered as a cyclic word). Then 
    \[ |u| \ \le \ \frac{|w|}{2n}. \]
  \end{keylem}

  \begin{corollary}\label{Cor:1}
    Let $(S, \{M_a\}_{a \in V(\Gamma)})$ be a taut surface with pattern for
    $g^n$. Then the total number of bands on $S$ is at least $n$.
  \end{corollary}

  \begin{proof}
    Recall from \remref{arc-equiv} that the bands provide a partition
    of the set of pattern arcs. The subwords $u, u^{-1}$ of
    $w_{\partial S}$ associated to bands then form a partition of the
    letters of $w_{\partial S}$. Each band accounts for two subwords,
    each of length at most $|w_{\partial S}|/(2n)$, so there must be
    at least $n$ bands. 
  \end{proof}

\section{Proof of~\thmref{1}}\label{Sec:Proof1}

  Let $G = A(\Gamma)$ and suppose $g^n \in [G,G]$ for some $n>0$. Let
  $S_0 \to \fat(\Gamma)$ be an admissible surface for $g^n$. By
  \propref{taut} there is a taut surface with pattern $(S, \{M_a\}_{a
  \in V(\Gamma)})$ for $g^n$ such that $1 - \chi(S_0) \geq 1 -
  \chi(S)$.  

  Recall that the \emph{chromatic number} of $\Gamma$ is the smallest 
  number of colors needed to color the vertices of $\Gamma$ so that no
  two adjacent vertices have the same color. Let such a
  coloring with $k$ colors be given. Every pattern arc inherits a
  color from its transverse label, and whenever two pattern arcs
  cross, they must have different colors. In particular, the set of
  pattern arcs having any single color is pairwise disjoint. 

  For each band in $S$, assign it the color of one of its pattern
  arcs. Since there are at least $n$ bands (\corref{1}) and only $k$
  colors, there must be at least $n/k$ bands of the same color $c$,
  for some $c$ (pigeonhole principle). Taking one pattern arc with
  color $c$ from each of these bands, we obtain a collection
  $\mathcal{C}$ of pattern arcs that is pairwise disjoint and
  non-parallel, of size at least $n/k$. 

  The size of any such collection is at most $6\genus(S) - 3$. To
  see this, enlarge $\mathcal{C}$ to a maximal such collection
  $\mathcal{C'}$, which defines an ideal triangulation of $S
  - \partial S$ (or equivalently, a one-vertex triangulation of
  $S/\partial S$). Such a triangulation has $6\genus(S) -3$
  edges. 

  Since
  \[
    \frac{n}{k} \ \leq \ 6\genus(S) - 3
  \]
  we have
  \[
    \genus(S_0) \ \ge \ \genus(S) \ \ge \ \frac n{6k}+\frac12.
  \]
  The right hand side is a lower bound for $\cl(g^n)$, since
  $S_0$ was an arbitrary admissible surface for $g^n$. 
  Dividing by $n$ and taking ${n\to\infty}$, we obtain
  \[
    \scl(g) \ \ge \ \frac 1{6k}.
  \]
  This finishes the proof of~\thmref{1}.

\section{Proof of~\thmref{2}}

  \label{Sec:Proof2}

  Let $G = A(\Gamma)$, where $\Gamma$ has no triangles, and suppose $g^n \in
  [G,G]$ for some $n>0$. As before, let $S_0 \to \fat(\Gamma)$ be an
  admissible surface for $g^n$. By \propref{taut} there is a taut
  surface with pattern $(S, \{M_a\}_{a \in V(\Gamma)})$ for $g^n$ such
  that $1 - \chi(S_0) \geq 1 - \chi(S)$. 

  By property \ref{t5}, $S$ can be given the structure of a combinatorial
  $2$--complex, with $0$--cells equal to the intersection points and
  $1$--skeleton equal to $\bigcup_{a \in V(\Gamma)} M_a \cup \partial S$. 
  Each $2$--cell, or face, is a polygon as described in Definition
  \ref{Def:polygon}. These faces each have four or more sides by property
  \ref{t6}. We further endow $S$ with the structure of an \emph{angled
  $2$--complex} by declaring that every corner has angle $\pi/2$. 


  Observe that the curvature $\k(v)$ of every $0$--cell is $0$. Indeed, for
  an interior $0$--cell $v$ we have: $\k(v)=2\pi-\pi\cdot
  0-4\cdot\frac\pi2=0$, and for a $0$--cell $v$ on the boundary $\dS$ we
  have: $\k(v)=2\pi-\pi\cdot 1-2\cdot\frac\pi2=0$. Thus, the combinatorial
  Gauss--Bonnet formula gives us: 
  \[
  2\pi\chi(S) \ = \ \sum_{f\in\text{$2$--cells}}\k(f).
  \]

  For a $2$--cell $f$ we have 
  \[
  \k(f) \ = \ \frac\pi2\big(\text{\# of corners of
    $f$}\big)-\big((\text{\# of sides of $f$})-2\big)\pi \ = \ 2 
  \pi-\frac\pi2\big(\text{\# of sides of $f$}\big). 
  \]
  Therefore: 
  \begin{equation*}\label{eq1}
  1 - \chi(S) \ = \
  1+\sum_{f\in\text{$2$--cells}}\frac14\big((\text{\# of sides of
    $f$})-4\big).\tag{$\ast$} 
  \end{equation*}

  Recall that all faces of $S$ have four or more
  sides. The faces with exactly four sides contribute $0$
  to the sum in~\eqref{eq1}. Thus one can sum over only the $2$--cells
  $f$ which have $\ge5$ sides. For simplicity in what follows we will
  call them \emph{special faces}. 

  Our goal now is to relate the quantity in~\eqref{eq1} to the
  number of bands on $S$. We will do this simultaneously for the 
  case of free groups (thus obtaining Culler's bound $\scl(g) \ge
  1/6$) and for the case of two-dimensional right-angled Artin groups.  

  First, we modify the right-hand side of~\eqref{eq1} as follows. In the
  case of free groups the pattern arcs do not intersect each
  other. Hence the sides of a $2$--cell alternate between pattern arcs 
  and arcs in $\dS$. Therefore, the number of sides
  of any $2$--cell is always even. Hence the minimal number of sides of
  a special face is $6$. In the right-angled Artin group case, special
  faces can have $5$ or more sides. Thus we have: 
  \[
  (\text{\# of sides of $f$})-4 \ \ge \ A\cdot(\text{\# of sides of $f$}),
  \]
  where $A=\dfrac13$ for free groups and $A=\dfrac15$ for RAAGs.

  Second, recall from \remref{band-special-face} that each band is
  adjacent to at least one special face on each of its two sides. It is
  possible that the two sides of the band are adjacent to the
  \emph{same} special face, but they will do so in distinct sides of that
  face. Thus we may count the sides of special faces as follows: 
  \[
  \sum_{f\in\text{special faces}}(\text{\# of sides of $f$}) \ \ge \
  (\text{\# of bands})\times 2. 
  \]
  This inequality can be strict if there is more than one special face on
  one side of a band, or if there is a special face with one or
  more sides lying on $\dS$. Then these sides will not be accounted for
  by bands. 

  In the free group case, each special face has exactly half of its
  sides lying on $\dS$, so bands border exactly half of the total count
  of the sides of special faces. Thus we have: 
  \[
  \sum_{f\in\text{special faces}}(\text{\# of sides of $f$}) \ \ge \
  (\text{\# of bands})\times 4. 
  \]
  Going back to formula~\eqref{eq1}, we get:
  \[
  1 - \chi(S) \ \ge \ 1+\frac B4 \big(\text{\# of bands}\big),
  \]
  where $B=\dfrac 13\cdot 4$ for free groups and $B=\dfrac 15\cdot2$ for
  RAAGs. 

  Corollary~\ref{Cor:1} tells us that $(\text{\# of bands})\ge n$. Therefore
  \[
  \genus(S_0) \ = \ \dfrac12\cdot\big(1 - \chi(S_0)\big) \ \geq \
  \dfrac12\cdot\big(1 - \chi(S)\big) \ \ge \  
  \begin{cases}
  \ \dfrac12+\dfrac n{6}& \text{for free groups,}\\[1em]
  \ \dfrac12+\dfrac n{20}& \text{for RAAGs.}\\
  \end{cases}
  \]
  The right hand side is now a lower bound for $\cl(g^n)$ as before. 
  Dividing by $n$ and taking $n\to\infty$ yields
  \[
  \scl(g) \ \ge \ 
  \begin{cases}
    \ \dfrac 1{6}& \text{for free groups,}\\[1em]
    \ \dfrac 1{20}& \text{for RAAGs.}\\
  \end{cases}
  \]

  This finishes the proof of \thmref{2}.

\section{The non-overlapping property}

  \label{Sec:NoOv}

  The purpose of this section is to prove \thmref{NoOverlapping}, which was
  the key ingredient for the estimation of $\scl$ in the previous section
  (see \corref{1}). We start with basic definitions and some additional
  notions from~\cite{FFT}.

  \subsection*{Preliminaries}

  A CAT(0) cube complex $Y$ is a simply connected polyhedral complex in
  which the closed cells are standard Euclidean cubes $[0, 1]^n$ of various
  dimensions, such that any two cubes either have empty intersection or
  intersect in a single face, and the link of every vertex is a flag
  complex. The latter condition, called the Gromov Link Condition,
  guarantees that $Y$ is non-positively curved. The dimension of $Y$ is the
  dimension of its maximal dimensional cube. 

  An $n$--cube $[0,1]^n$ has $n$ midcubes defined by setting one of the
  coordinates equal to $1/2$. A \emph{hyperplane} in $Y$ is a closed
  subspace whose intersection with each cube is either empty or equal
  to a midcube. Each hyperplane separates $Y$ into two connected
  components. The closure of either of these two components is called
  a \emph{half-space}. The hyperplane which bounds a half-space $H$ is
  denoted by $\partial H$ and the half-space complementary to $H$ is
  denoted by $\ov{H}$. The set $\H(Y)$ of half-spaces of $Y$ is
  partially ordered by inclusion. Two distinct half-spaces $H$ and $H'$ are
  \emph{nested} if either $H \supset H'$ or $H' \supset H$; they are
  \emph{tightly-nested} if they are nested and no other half-space is
  nested between them. We say that $H$ and $H'$ \emph{cross}, denoted
  $H \trans H'$, if $\partial H$ and $\partial H'$ intersect. When
  this occurs, there is a square $S$ in $Y$ in which $\partial H \cap
  S$ and $\partial H' \cap S$ are the two midcubes of $S$. 

  Given $H, K \in \H(Y)$, with $H \supset K$, we will call a sequence of
  half-spaces $\gamma = \set{H_0, H_1,\ldots,H_n, H_{n+1}}$ a \emph{chain}
  of length $n$ from $H$ to $K$ if 
  \[ 
  H=H_0 \supset H_1 \supset \cdots \supset H_n \supset H_{n+1} = K.
  \] 
  A chain is \emph{taut} if every adjacent pair is tightly-nested. A chain
  is \emph{longest} if $n$ is largest possible; such chain is necessarily
  taut. Note that any chain from $H$ to $K$ can be enlarged to be a taut
  chain. We will call $H_m$, where $m = n/2$ if $n$ is even or $m=(n+1)/2$
  if $n$ is odd, the \emph{midpoint} of the chain. We make the following
  observation. 
  
  \begin{lemma} \label{Lem:MaxChains}
    Suppose $\gamma = \set{H,H_1,\dotsc,H_n,K}$ and $\gamma' = \set{H,
    H_1', \dotsc, H_n',K} $ are two longest chains from $H$ to $K$. Then
    the midpoints of $\gamma$ and $\gamma'$ either cross or coincide. 
  \end{lemma}
  
  \begin{proof}
    Let $H_m$ and $H_m'$ be the midpoints of $\gamma$ and $\gamma'$
    respectively. If $H_m \ne H_m'$ and they do not cross, then the only
    other possibility is that they are nested. If $H_m \supset H_m'$ then
    the chain 
    \[\set{H,H_1,\ldots,H_m, H_m', \ldots, H_n',K}\] 
    from $H$ to $K$ is strictly longer than $\gamma$, a contradiction. If
    $H_m' \supset H_m$ then again one can construct a chain from $H$ to $K$
    that is longer than $\gamma$. 
  \end{proof}

  Equip $Y^{(1)}$ with the edge path metric $d$. Given two vertices
  $x$ and $y$, define the following set of half-spaces: 
  \[ [x,y] = \set{ H \in \half(Y) \st x \notin H, y \in H}.\] 
  An edge path from $x$ to $y$ is a geodesic if and only if it does
  not cross any hyperplane twice. Thus $d(x,y) = \card{[x,y]}$, where
  $\card{[x,y]}$ is the cardinality of $[x,y]$. An essential feature
  of this distance function is that $(Y,d)$ is a \emph{median
  space} \cite{Gerasimov,Roller}. That is, for every triple of
  vertices $x, y, z$, there 
  exists a unique vertex $m=m(x,y,z)$ such that $d(a,b) = d(a,m) +
  d(m,b)$ for all distinct $a,b \in \set{x, y, z}$. Equivalently,
  $[a,b]$ is the disjoint union $[a,m] \cup [m,b]$ for all distinct
  $a, b \in \set{x,y,z}$. 

  We say that $H \in \half(Y)$ \emph{intersects} an edge $e$ of $Y$ if
  $\partial H$ intersects $e$. Every oriented edge $e=(x,y)$ in $Y$
  naturally defines a half-space, namely $H=[x,y]$; and conversely, every
  half-space naturally defines an orientation on all the edges it
  intersects. 

  We say that $H \in \half(Y)$ \emph{intersects} a subcomplex $Z \subset Y$
  if $H$ intersects an edge of $Z$. If $Z$ is connected, then whenever $H,
  K \in \half(Y)$ both intersect $Z$ and are nested, every half-space in
  between them also intersects $Z$. A full subcomplex $Z$ of $Y$ is called
  \emph{convex} if $Z^{(1)}$ is convex in $Y^{(1)}$ with respect to $d$;
  that is, $Z$ contains all edge geodesics between pairs of vertices in
  $Z$. 

  \subsection*{The cube complex of a right-angled Artin group}

  Let $A(\Gamma)$ be a right-angled Artin group and $X(\Gamma)$ the 
  Salvetti complex associated to $A(\Gamma)$. By equipping each torus with
  the standard Euclidean metric, we obtain a CAT(0) cube complex structure
  on the universal cover $Y(\Gamma)$ of $X(\Gamma)$. The dimension of
  $Y(\Gamma)$ is the dimension of a maximal dimensional torus in
  $X(\Gamma)$. Equivalently, this is the largest size of a complete
  subgraph of $\Gamma$. 
  
  The $2$--skeleton of $Y(\Gamma)$ is the Cayley $2$--complex of the
  defining presentation for $A(\Gamma)$. That is, every oriented edge of
  $Y(\Gamma)$ is labeled by a generator of $A(\Gamma)$ or its inverse, and
  squares are bounded by edges whose labels correspond to \emph{distinct}
  vertices of $\Gamma$ that bound an edge. For any $H \in \half(Y)$, every
  oriented edge intersected by $H$ has the same label, and therefore every
  $H \in \half(Y)$ can be labeled accordingly. If $H \in \half(Y)$ has
  label $a$, then $\ov{H}$ has label $a^{-1}$. Note that $A(\Gamma)$ acts
  on $\half(Y)$ in a label-preserving manner. 

  Given $H, K \in \half(Y)$, if $H \trans K$, then this occurs in a
  square, and their labels correspond to distinct adjacent vertices of
  $\Gamma$. If $H, K$ are tightly-nested, then there exists an 
  oriented edge path of length two whose edges are dual to them, and
  these edges do not bound a square. Thus their labels either are equal, or
  correspond to distinct vertices of $\Gamma$ that are not adjacent. 

  These observations lead easily to the following results. They are
  a reformulation of the axioms of \emph{special cube
  complexes} from \cite{HaglundWise}, as stated in \cite[Lemma~7.3]{FFT}. 

  \begin{proposition} \label{Prop:Special}
    The action of $A(\Gamma)$ on $Y(\Gamma)$ satisfies the following
    properties: 
    \begin{enumerate}
      \item \label{s1} There do not exist $H\in \half(Y)$, $f \in
        A(\Gamma)$ such that $f(\ov{H}) = H$.
      \item \label{s2} There do not exist $H \in \half(Y)$, $f \in
        A(\Gamma)$ such that $H \trans fH$.
      \item \label{s3} There do not exist $H \in \half(Y)$, $f \in
        A(\Gamma)$ such that $H$ and $f(\ov{H})$ are tightly-nested.
      \item \label{s4} There do not exist a tightly-nested pair $H, K \in
        \half(Y)$ and $f \in A(\Gamma)$ such that $H \trans f(K)$. \qed
    \end{enumerate}
  \end{proposition}

  As discussed in {\cite[Section 3]{FFT}, 
  every non-trivial $g \in A(\Gamma)$ has is a \emph{combinatorial axis}
  $L$, which is a bi-infinite geodesic in the $1$--skeleton of $Y(\Gamma)$
  such that $g(L) = L$. Let $A_g$ be the set of half-spaces that intersect
  $L$. Though combinatorial axes for $g$ are not unique, $A_g$ is
  independent of the choice of $L$. There is a decomposition of $A_g$ into
  two disjoint collections $A_g = A_g^+ \sqcup A_g^-$ such that $H \in
  A_g^-$ if and only if $\ov{H} \in A_g^+$, and $H \supset gH$ for all $H
  \in A_g^+$. In other words, every $H \in A_g^+$ contains the positive
  (attracting) end of $L$ and $\ov{H}$ contains the negative end.
  Hence, for every pair $H, K \in A_g^+$, either $H \trans K$ or $H$ and
  $K$ are nested. 

  Every non-trivial $g \in A(\Gamma)$ has an \emph{essential characteristic
  set} $Y_g \subset Y(\Gamma)$. 
It
  is a convex $\langle g\rangle$--invariant subcomplex, intersected by all the
  half-spaces of $A_g$ and no others.

  \begin{proposition}[{\cite[Section 3]{FFT}}] \label{Prop:Essential}
    The essential characteristic set $Y_g \subset Y(\Gamma)$ satisfies the
    following properties: 
    \begin{enumerate}
      \item For any $n > 0$, if $w$ is a cyclically reduced word
        representing the conjugacy class of $g^n$ and $x$ is a vertex
        of $Y_g$, then 
        \[ |w| \ = \ d(x,g^nx) \ = \ n \, d(x,gx).\] 
        Moreover, there
        is a combinatorial axis $L \subset Y_g$ for 
        $g^n$, such that every cyclic conjugate of $w$ appears as a
        word along the labels of $L$ (read in the positive direction). The set of half-spaces
        crossing $L$ and containing the positive end of $L$ is $A_g^+$. 
      \item For every $x \in Y_g$ we have $\displaystyle A_g^+ =
        \bigcup_{n \in \Z} [g^n x , g^{n+1}x]$. \qed
    \end{enumerate}
  \end{proposition}

  \begin{lemma} \label{Lem:NoOverlapping}
    Given $x, y \in Y_g$ with $[x,y] \subset A_g^+$ and $\card{[x,y]} >
    \card{[x,gx]}/2$, if $[fy,fx] \subset A_g^+$ for some $f \in
    A(\Gamma)$, then there exists $H \in [x,y]$ and $n \in \Z$ such that
    $g^n H \in [fy,fx]$.
  \end{lemma}

  \begin{proof}
    In the following, by a \emph{copy} of a half-space $H \in A_g^+$, we
    will mean $g^kH$ for some $k \in \Z$. Note that since $[fy,fx] \subset
    A_g^+$, each element of $[fy, fx]$ is a copy of some element in $[x,gx]$. We proceed
    by considering two cases. 
    
    First suppose $[x,y] \subseteq [x,gx]$. In this case, $y=m(x,y,gx)$, so
    $\card{[x,y]} + \card{[y,gx]} = \card{[x,gx]}$. Since $\card{[x,y]} >
    \card{[x,gx]}/2$, \card{[y,gx]} < \card{[x,y]} = \card{[fy,fx]}. Thus, if
    $[fy,fx]$ contains at most one copy of each element of $[y,gx]$, then it
    must contain a copy of some element in $[x,y]$ and the statement holds.
    So suppose $[fy,fx]$ contains at least two copies of some $K \in
    [y,gx]$. That is, there exist $i < j$ such that $g^i K, g^j K \in
    [fy,fx]$. By definition, $g^i K = f \ov{H}$ for some $H \in
    [x,y]$. Consider $g^{i+1} H \in [g^{i+1}x, g^{i+1}y]$. The
    half-spaces $g^i K$, $g^{i+1}H$, and $g^j K$ have the same labels
    (up to sign), so no two of them can cross. Hence they are nested
    in some linear order. Note that $g^{i+1}x\in g^i K - g^{i+1}H$ and 
    $g^j y \in g^{i+1}H - g^j K$, since $i+1 \le j$. Thus $g^i K \supset
    g^{i+1}H \supset g^j K$. But this implies $g^{i+1}H \in [fy,fx]$,
    concluding the argument in this case. 

    Now suppose that $[x,y] \not\subseteq [x,gx]$. We
    can also assume that $[x,gx] \not\subseteq [x,y]$, for otherwise
    we will already be done. Let $z$ be the median of $x$, $y$, and
    $gx$. Then $[x,gx] = [x,z] \cup [z,gx]$ and $[x,y] = [x,z] \cup
    [z,y]$, where $[z, gx]$ and $[z, y]$ are both non-empty. Let $K \in [z,y]$ be any
    element. For any $K' \in [z,gx]$, since $K$ and $K'$ both lie in
    $A_g^+$, they are either nested or they cross. But $gx \in K'$ and
    $gx \notin K$ and $y \in K$ and $y \notin K'$, so it is impossible
    for them to be nested. Therefore, $K$ crosses every element of
    $[z,gx]$. Now consider $f\ov K \in [fy,fx]$, which by assumption
    lies in $A_g^+$, and thus $f \ov K = g^n H$ for some $H \in
    [x,gx]$. If $H \in [z,gx]$ then $g^n H$ and $g^n K$ must cross, as
    $H$ and $K$ do, but this is not possible since $g^n H = f\ov
    K$. Therefore $H \in [x,z] \subset [x,y]$, which concludes the
    proof. 
  \end{proof}
   
  \begin{proposition} \label{Prop:NoOverlapping}
    Let $x, y \in Y_g$ be vertices such that $[x,y] \subset A_g^+$ and
    $\card{[x,y]} > \card{[x,gx]}/2$. Then there does not exist $f \in
    A(\Gamma)$ such that $[fy,fx] \subset A_g^+$. 
  \end{proposition}

  \begin{proof}
    The proof is by contradiction. Suppose such an element $f$ exists. Note
    that if $[fy,fx] \subset A_g^+$, then $[g^nfy,g^nfx] \subset A_g^+$ for
    all $n \in \Z$. Thus, using \lemref{NoOverlapping} and replacing $f$ by
    $g^{-n}f$ if necessary, we may assume that there exist $H, K \in [x,y]$
    such that $H = f \ov K$. It is not possible that $H = K$ or $H
    \trans K$, by \propref{Special} (\ref{s1}) and (\ref{s2}), so $H$
    and $K$ must 
    be nested. Suppose $H \supset K$. First we claim that $f \ov H = K$.
    Consider a longest chain from $H$ to $K$: 
    \[ 
      \gamma = \set{H, H_1, \ldots, H_n, K}.
    \]
    Maximality and nesting are preserved by the action of $A(\Gamma)$, so 
    \[ 
      f \ov{\gamma} \ = \ \set{ f \ov K, f\ov{H}_n, \dotsc, f\ov{H}_1, f \ov{H}} 
      \ = \ \set{ H, f\ov{H}_n, \dotsc, f\ov{H}_1, f \ov{H}} 
    \]
    is a longest chain from $H$ to $f \ov H$. If $f \ov{H} \ne K$, then
    they must be nested. If $f \ov{H} \supset K$, then 
    \[ \set{H, f\ov{H}_n, \dotsc, f\ov{H}_1, f \ov{H}, K} \] is strictly
    longer than $\gamma$, contradicting the choice of $\gamma$. Similarly,
    if $K \supset f \ov H$, then 
    \[ 
     \set{ H, H_1, \ldots, H_n, K, f \ov{H}} \ = \  
     \set{ f\ov{K}, H_1, \ldots, H_n, K, f \ov{H}} 
    \] 
    is strictly longer than $f \ov{\gamma}$. This shows that $f\ov H = K$.
    In particular, both $\gamma$ and $f\ov{\gamma}$ are longest chains from
    $H$ to $K$. To proceed with the contradiction, let $H_m$ and $H_m'$ be
    the midpoints of $\gamma$ and $f\ov{\gamma}$ respectively. By
    \lemref{MaxChains}, either $H_m = H_m'$ or they cross. If $n$ is odd,
    then $H_m' = f \ov{H}_m$. In this case, if $H_m = H_m'$, then this
    violates (\ref{s1}) of \propref{Special}. If $H_m$ and $H_m'$ cross,
    then this violates (\ref{s2}) of \propref{Special}. If $n$ is even,
    then $H_m' = f \ov{H}_{m+1}$. In this case, if $H_m = H_m'$, then
    $f^{-1}(\ov{H}_m) = H_{m+1}$. Since $H_m$ and $H_{m+1}$ are
    tightly-nested, this violates (\ref{s3}). Finally, if $H_m$ and $H_m'$
    cross, then since $H_m$ and $H_{m+1}$ are tightly-nested, this violates
    (\ref{s4}).
    
    The case $K \supset H$ will yield a similar contradiction. This
    concludes the proof. 
  \end{proof}

  \begin{theorem} \label{Thm:NoOverlapping} 
    Given $g\in A(\Gamma)$ and $n>0$, let $w$ be a cyclically reduced word
    in the generators of $A(\Gamma)$ representing the conjugacy class of
    the element $g^n$ in $A(\Gamma)$, and suppose $u$ is a word such that
    both $u$ and $u^{-1}$ appear as subwords of $w$ (considered as a cyclic
    word). Then  \[|u| \le \frac{|w|}{2n}.\]
  \end{theorem}
  
  \begin{proof}
    \thmref{NoOverlapping} follows easily from \propref{NoOverlapping}. To
    see this, let $Y_g$ be the essential characteristic set for $g$. By
    \propref{Essential} there is a path $P$ in $Y_g$ whose labels read $u$.
    Let $x$ and $y$ be, respectively, the initial and terminal endpoints of
    $P$, and let $\ov{P}$ be the reversal of $P$ (which reads $u^{-1}$).
    Since $w$ represents a positive power of $g$, $[x,y] \subset A_g^+$. 

    By the same reasoning there is a path $P'$, from $x'$ to
    $y'$, whose labels read $u^{-1}$, satisfying $[x', y'] \subset
    A_g^+$. The action of $A(\Gamma)$ is transitive on vertices,
    so there exists $f \in A(\Gamma)$ such that $fy = x'$. Since
    $\ov{P}$ and $P'$ both read $u^{-1}$, we must have $f\ov{P} = P'$,
    and therefore $[x',y'] = [fy,fx]$. Using \propref{NoOverlapping}
    we conclude that 
    \[ |u| \ \le \ \frac{d(x,gx)}{2} \ = \ \frac{|w|}{2n}. \qedhere \]
  \end{proof}


  \bibliographystyle{amsalpha}
  \bibliography{current}  

  \end{document}